\documentclass[12pt]{amsart}

\usepackage{amsmath,amsthm,amssymb,amsfonts}
\usepackage{latexsym}
\usepackage[all]{xy}
\usepackage{tikz}   

\newtheorem{thm}{Theorem}[section]
\newtheorem{cor}[thm]{Corollary}
\newtheorem{lem}[thm]{Lemma}
\newtheorem{prop}[thm]{Proposition}
\theoremstyle{remark}
\newtheorem{rem}[thm]{Remark}
\theoremstyle{definition}
\newtheorem{defn}[thm]{Definition}
\newtheorem{ex}[thm]{Example}
\newtheorem*{excont}{Example \ref{ex:pentagon} continued}

\numberwithin{equation}{section}

\title[Graph products]{Graph products of spheres, associative
  graded algebras and Hilbert series}

\date{\today}

\author{Peter Bubenik}
\address{Cleveland State University}
\email{p.bubenik@csuohio.edu}

\author{Leah H. Gold}
\address{Cleveland State University}
\email{l.gold33@csuohio.edu}

\thanks{}

\newcommand{\bfk}{{\mathbf k}}
\newcommand{\freeL}{{\mathbb L}}
\newcommand{\ZZ}{{\mathbb{Z}}}
\newcommand{\Z} {\mathbb{Z}}
\newcommand{\N}{\mathbb{Z}_{>0}}
\newcommand{\ZM}{\ZZ\langle\langle M \rangle\rangle}
\newcommand{\Zz}{\ZZ\langle\langle z \rangle\rangle}
\newcommand{\Zmodtwo} {\mathbb{Z}/2\mathbb{Z}}
\newcommand{\from}{\leftarrow}

\newcommand{\AH}{\mathbf{AH}}

\providecommand{\abs}[1]{\lvert#1\rvert}
\providecommand{\isom}{\cong}
\providecommand{\incl}{\hookrightarrow}

\DeclareMathOperator{\initial}{in}
\DeclareMathOperator{\weight}{wgt}
\DeclareMathOperator{\Weight}{Wgt}
\DeclareMathOperator{\Tor}{Tor}
\DeclareMathOperator{\DGA}{DGA}

\begin{document}

\begin{abstract}
  Given a finite, simple, vertex--weighted graph, we construct a
  graded associative (non-commutative) algebra, whose generators
  correspond to vertices and whose ideal of relations has generators
  that are graded commutators corresponding to edges.  We show that
  the Hilbert series of this algebra is the inverse of the clique
  polynomial of the graph. Using this result it easy to recognize if
  the ideal is \emph{inert}, from which strong results on the algebra
  follow.  Noncommutative Gr\"obner bases play an important role in
  our proof.

  There is an interesting application to toric topology. This algebra
  arises naturally from a partial product of spheres, which is a
  special case of a generalized moment--angle complex. We apply our
  result to the loop--space homology of this space.
\end{abstract}

\maketitle
\section{Introduction}\label{section:introduction}
This paper connects ideas from algebra and algebraic topology and
tries to provide sufficient background to be accessible to both
audiences.

Let $\Gamma$ be a finite simple graph with vertices $V$ and edges $E$,
in which each vertex $i \in V$ is labeled with a positive integer
$p_i$, called the {\it weight} of the vertex $i$. For $j \in E$, let
$a_j, b_j$ denote its endpoints. Let $c_{i,j}$ be the number of
complete subgraphs of $\Gamma$ with $i$ vertices whose weights sum to
$j$. Call
\begin{equation*}
   c_{\Gamma}(z) =  \sum\limits_{i=0}^{\infty}\sum\limits_{j=0}^{\infty}
    (-1)^i c_{i,j} z^j
\end{equation*}
the \emph{clique polynomial} of the weighted graph $\Gamma$. It is a
polynomial because $\Gamma$ is finite.  Let $\bfk$ be a field with
characteristic not equal to $2$. Let $[a,b]$ denote the graded
commutator $ab - (-1)^{\abs{a}\abs{b}}ba$.  Our main result, which is
closely related to a similar result of Cartier and
Foata\cite{cartierFoata}, is the following.

\begin{thm} \label{thm:main}
The associative (noncommutative) graded algebra
 \[
 A(\Gamma) = {\bfk} \langle x_i \; , i \in V \rangle / I(\Gamma),
\]
where $I(\Gamma)$ is the two--sided ideal $([x_{a_j},x_{b_j}],\; j\in E)$,
   with $\abs{x_i} = p_i$  has Hilbert series
  \[
  H_{A(\Gamma)}(z) = \left[c_{\Gamma}(z)\right]^{-1}.
  \]
\end{thm}
We remark that even though $[x_{a_j},x_{b_j}]$ may depend on the
ordering of the endpoints of the edge $j\in E$, the ideal $I(\Gamma)$
does not.

From this theorem we obtain the following corollary. For $j\in E$, let
$q_j = p_{a_j} + p_{b_j}$.  For associative algebras, $A$ and $B$, we
write $A \amalg B$ for their free product.
Let $B(\Gamma)$ denote the subalgebra of the free associative algebra $\bfk \langle x_i, i \in V \rangle$ that is generated by $\{[x_{a_j},x_{b_j}], j \in E\}$.
\begin{cor} \label{cor:algebra}
  The following are equivalent.
  \begin{enumerate}
  \item The graph $\Gamma$ does not contain a triangle, i.e. a $3$--cycle.
  \item $ \bfk \langle x_1, \ldots, x_n \rangle \isom B(\Gamma) \amalg
    A(\Gamma) \text{ (as vector spaces),} $
  \item $ H_{A(\Gamma)}(z) = \dfrac{1}{1-(z^{p_1} + \cdots + z^{p_{\abs{V}}}) +
      (z^{q_1} + \cdots + z^{q_{\abs{E}}})}, $ and
  \item $A(\Gamma)$ has global dimension $\leq 2$.
  \end{enumerate}
\end{cor}
If these equivalent conditions are satisfied, we say that the
two--sided ideal, $I(\Gamma)$, is \emph{inert}.
Equivalently, we say that the set $\{[x_{a_j},x_{b_j}], \; j \in E\}$ is inert, or that $A(\Gamma)$ has a finite inert presentation.

\bigskip

Inert sets play a central role for non--commutative algebra analogous
to regular sequences in commutative ring theory. This analogy is made
precise by Anick\footnote{In \cite{anick:inert}, Anick uses ``strongly
  free'' in place of ``inert''.}~\cite{anick:inert}.  In general, most
problems for finitely presented associative algebras are
unsolvable. For example, the word problem is
undecidable~\cite{Tsejtin}.  However, if the set of relations is
inert, then the situation is as simple as it can be, and quite a lot
can be said. Unfortunately, it can be very hard to determine whether
or not a given ideal is inert. Anick~\cite{anick:inert} gives some
sufficient conditions, but they do not apply to most of the examples
considered here. A simple non--trivial example is $A(\Gamma)$, where
$\Gamma$ is the pentagon.

This paper gives a large class of associative algebras, $A(\Gamma)$, with finite presentations, for which we can easily check inertness. We analyze the case where $\Gamma$ is the pentagon in detail. Another example, which is easily seen to be inert from its clique polynomial, is $A(\Gamma)$, where $\Gamma$ is the one--skeleton of the dodecahedron. The authors know of no other way to ascertain this fact. 
Furthermore, we obtain a large class of finitely presented algebras whose presentations are not inert, but for which we can nevertheless easily calculate their Hilbert series.
In addition, we show that these algebraic results have topological versions.

\bigskip

The algebra $A(\Gamma)$ also arises from a ``graph product'' of
spheres. This is an example of a generalized moment--angle complex,
which we now describe.

Let $K$ be a simplicial complex with vertices $\{1, \ldots, n\}$.  Let
$\underline{X}$ be the collection
\begin{equation*}
  \underline{X} = \left\{ (X_i, A_i) \right\}_{i=1}^n,
\end{equation*}
where $A_i \subseteq X_i$ is a pair of topological spaces.

For each face $\sigma \in K$, define
\begin{equation*}
  X^{\sigma} = \Bigl\{ (x_1, \ldots, x_n) \in \prod_{i=1}^n X_i
  \ \mid \ x_i \in A_i \text{ if } i \notin \sigma \Bigr\} = \prod_{i=1}^n Y_i, 
\end{equation*}
\begin{equation*}
\text{ where }Y_i =
\begin{cases}
  X_i & \text{if } i \in \sigma\\
  A_i & \text{if } i \notin \sigma .
\end{cases}
\end{equation*}
Then the \emph{generalized moment--angle complex} is given by
\begin{equation*}
  \underline{X}^K = \cup_{\sigma \in K} X^{\sigma}.
\end{equation*}

In the case in which $(X_i,A_i) = (D^2,S^1)$ for all $i$, one obtains
the moment--angle complex defined by Davis and Januszkiewicz
\cite{davisJanuszkiewicz}, and studied in detail by Buchstaber and
Panov~\cite{bukhshtaberPanov:actionsOfTori}, Notbohm and
Ray~\cite{notbohmRay} and Grbi{\'c} and
Theriault~\cite{grbicTheriault}. Davis and Januszkiewicz showed that
every smooth projective toric variety is the quotient of a
moment--angle complex by the free action of a real torus.  The above
generalization was defined by Strickland~\cite{strickland:toricSpaces}
and has been studied by Denham and
Suciu~\cite{denhamSuciu:momentAngleComplexes}, Bahri, Bendersky, Cohen
and Gitler~\cite{bbcg:polyhedralProduct} and F\'elix and
Tanr\'e~\cite{felixTanre:polyhedralProduct}.  This construction is
also called a \emph{partial product}, or a \emph{polyhedral
  product}. The notation $Z(K,\underline{X})$ is also used for
$\underline{X}^K$.

We are interested in the case where $\underline{X}$ is a collection of
pointed spheres:
\begin{equation*}
  \underline{X} = \left\{ (S^{p_i+1},*) \right\}_{i=1}^n,  
\end{equation*}
where $*$ is the one--point space with fixed inclusion $* \incl
S^{p_i+1}$, and $K$ is one-dimensional. That is, $K$ is a simple graph
$\Gamma$ with $n$ vertices.

Let $E=\{1,\ldots,m\}$ be the set of edges of $\Gamma$.  For $j \in
E$, let $a_j,b_j$ be the vertices that $j$ connects.  Let $q_j =
p_{a_j} + p_{b_j}$.  Let $\alpha_j: S^{q_j+1} \to S^{p_{a_j}+1} \vee
S^{p_{b_j}+1}$ denote the top cell attachment of $S^{p_{a_j}+1} \times
S^{p_{b_j}+1}$ (given by the Whitehead product $[\iota_{a_j},
\iota_{b_j}]_W$ where $\iota_i$ is the identity map on $S^{p_i+1}$).
Then $\underline{X}^{\Gamma}$ is given by adjoining $m$ cells to a wedge of
$n$ spheres:
\begin{equation*}
  \underline{X}^{\Gamma} \isom \left( \vee_{i=1}^{n} S^{p_i+1} \right) \cup_f 
  \left( \vee_{j=1}^m D^{q_j+2} \right), \text{ where } f = \vee_{j=1}^m \alpha_j.
\end{equation*}
A space with such a construction is called a spherical two--cone.
Understanding spaces constructed using such attaching maps is a
nontrivial problem first studied by
J.H.C. Whitehead~\cite{whitehead:addingRelations}. More recent work
includes that by Halperin and Lemaire
\cite{lemaire:autopsie,halperinLemaire:inert},
Anick~\cite{anick:thesis}, F\'elix and
Thomas~\cite{felixThomas:attach}, and
Bubenik~\cite{bubenik:freeSemiInert,bubenik:separated}.

Let $Y = \underline{X}^{\Gamma}$, where $\underline{X} = \left\{
  (S^{p_1+1},*) \right\}_{i=1}^n$. Let $W = \underline{X}^{\Gamma_0} =
\vee_{i=1}^n S^{p_i+1}$ and let $i:W \hookrightarrow Y$ denote the
inclusion. Let $Z = \vee_{j=1}^m S^{q_j+1}$. So $Y = W \cup_f CZ$,
where $f = \vee_{j=1}^m \alpha_j: Z \to W$ denotes the attaching map
and $CZ$ denotes the cone on $Z$. Let $B(f)$ denote the image of
$H_*(\Omega f; \bfk): H_*(\Omega Z; \bfk) \to H_*(\Omega W;
\bfk)$. Let $I(f)$ denote the 2--sided ideal generated by $B(f)$. Let
$A(f) = H_*(\Omega W; \bfk) / I(f)$; that is, we have $A(f) \isom \bfk
\langle x_i \; , i \in V \rangle/([x_{a_j},x_{b_j}]\; , j \in E)$.

\begin{thm} \label{thm:topology}
The following are equivalent:
  \begin{enumerate}
  \item the graph $\Gamma$ does not contain a $3$--cycle,
  \item $H_*(\Omega i; \bfk): H_*(\Omega W; \bfk) \to H_*(\Omega Y;
    \bfk)$ is a surjection,
  \item $H_*(\Omega Y; \bfk) \isom A(f)$ (as vector spaces),
  \item $H_*(\Omega W; \bfk) \isom B(f) \amalg A(f)$,
  \item $A(f)$ has global dimension $\leq 2$, and
  \item $H_{A(f)}(z) = \dfrac{1}{1-(z^{p_1} + \cdots + z^{p_n}) +
      (z^{q_1} + \cdots + z^{q_m})}$.
  \end{enumerate}
\end{thm}
If these equivalent conditions are satisfied we say that the attaching
map $f$ is \emph{inert}. From these it follows that,
\begin{enumerate}
\item[(4')] $H_*(\Omega W; \bfk) \isom B(f) \amalg H_*(\Omega Y; \bfk)$ (as vector spaces),
\item[(5')] $H_*(\Omega Y; \bfk)$ has global dimension $\leq 2$, and
\item[(6')] $H_*(\Omega Y; \bfk)$ has Hilbert series
  $\frac{1}{1-(z^{p_1} + \cdots + z^{p_n}) + (z^{q_1} + \cdots +
    z^{q_m})}$.
\end{enumerate}

\bigskip

We have constructed two isomorphic algebras, $A(\Gamma)$ and
$A(f)$. Henceforth we denote them by $A$.
Anick~\cite{anick:homolAssocAlg} constructed a free graded
$A$-resolution of ${\bfk}$. It follows that homology groups
$\Tor_p^A({\bfk},{\bfk})$ are also graded.  Let
$\Tor_{p,q}^A({\bfk},{\bfk})$ denote the component of
$\Tor_p^A({\bfk},{\bfk})$ in grading $q$.  Since $A$ is finitely
generated, it is evident from Anick's resolution that
$\Tor_{p,q}^A({\bfk},{\bfk})$ is finite dimensional.  Thus one can
define the \emph{double Poincar\'e series} of $A$,
\begin{equation*}
  P_A(y,z) = \sum_{p,q \geq 0} \dim\left(\Tor_{p,q}^A({\bfk},{\bfk})\right) y^p z^q.
\end{equation*}
Anick~\cite[Equation (13)]{anick:homolAssocAlg} uses $\Tor_p^A({\bfk},{\bfk})$ to construct a minimal free $A$-resolution of ${\bfk}$, and calculates that
\begin{equation*}
  H_A(z) = \left[ P_A(-1,z) \right]^{-1}.
\end{equation*}
We thus have the following corollary.
\begin{cor}
$P_A(-1,z) = c_{\Gamma}(z)$.
\end{cor}

Since this paper contains results that may of interest to both
algebraists and topologists, we provide sufficient background for both
audiences.  We provide some background in
Section~\ref{section:background}, and show that
Corollary~\ref{cor:algebra} follows from Theorem~\ref{thm:main} using
results of Anick's on inert sets.  Our proof of Theorem~\ref{thm:main}
uses the theory of noncommutative Gr\"obner bases, which we introduce
in Section~\ref{section:grobner}, and a classical result of the theory
of free partially commutative monoids~\cite{cartierFoata}, which we
recall in Section~\ref{section:trace}. In Section~\ref{section:proof}
we prove a generalization of Theorem~\ref{thm:main}, in which we
assign weights to both the vertices and the edges of the graph. In
Section~\ref{section:topology} we apply our results to graph products
of spheres using Adams--Hilton models, obtaining
Theorem~\ref{thm:topology}.

\section{Background} \label{section:background}

\subsection{Monoids and formal power series}

Let $\langle x_1, \ldots, x_n \rangle$ denote the monoid generated by
$\{x_1, x_2, \ldots , x_n\}$ with the empty word 1 as unit.

\begin{defn} \label{def:fps} A \emph{formal power series} of a monoid
  $M$ is an element of the ring, $\ZM$, of functions (of sets) $M \to
  \Z$. A formal power series, $f$, is often denoted by $\sum_{w \in M}
  f(w) w$. 
\end{defn}

For example, a formal power series on $\langle z \rangle$, $f(k)=a_k$,
is denoted $\sum_{k=0}^{\infty} a_k z^k$.  The set $\ZM$ has the
structure of a ring under the usual addition and the non-commutative
Cauchy product.
\begin{gather*}
  (f+g)(w) = f(w) + g(w)\\
  (f \cdot g)(w) = \sum_{uv = w} f(u) g(v)
\end{gather*}
It has unit given by ${\bf1}(1) = 1$ and ${\bf1}(w) = 0$ if $w$ is nonempty. A
formal power series $f \in \ZM$ is invertible if and only if $f(1)$ is
invertible in $\Z$.  If $M = \langle x_1, \ldots, x_n \rangle$, we
will write $\ZM = \Z \langle \langle x_1, \ldots, x_n \rangle
\rangle$.

\subsection{Algebras and Hilbert series}

In this paper we will always assume $\bfk$ is a field with
characteristic not equal to $2$. Let $A$ be a (non-negatively) graded
associative $\bfk$-algebra. That is, $A = \oplus_{n=0}^{\infty} A_n$,
and for $x \in A_n$, $y \in A_m$, $xy \in A_{n+m}$. For $x \in A_n$,
we write $\abs{x} = n$.  We will always assume that our algebras are
non-negatively graded associative $\bfk$-algebras. Ideals will always
be assumed to be two-sided.

\begin{defn} \label{def:hilbertSeries} The {\it Hilbert series} of $A$
  is the formal power series in $\Zz$ given by
  $H_A(z)=\sum_{i=0}^\infty \dim_{\bfk} A_n \, z^n$. If $\dim (A_0) = 1$,
  then $H_A(z)$ is invertible, and we let $\left[H_A(z)\right]^{-1}$
  and $\frac{1}{H_A(z)}$ denote the inverse of $H_A(z)$.
\end{defn}

\begin{defn}
  For $x,y \in A$, the \emph{graded commutator} is given by $[x,y] =
  xy - (-1)^{\abs{x}\abs{y}}yx$.
\end{defn}

\subsection{Graphs and cliques}

Let $\Gamma$ be a simple graph (i.e. no double edges or loops) with a
finite set $V$ of vertices labeled in $\N$ and a set $E$ of edges. For
each vertex $i$ call its label $p_i$ its \emph{weight}. All graphs we
consider are assumed to be finite simple graphs whose vertices are
labeled in $\N$.

\begin{defn}
  An \emph{$i$--clique of weight $j$} of $\Gamma$ is a complete
  subgraph of $\Gamma$ on $i$ vertices whose weights sum to $j$. For
  $i,j \geq 0$ let $c_{i,j}(\Gamma)$ be the number of $i$--cliques of
  weight $j$ in $\Gamma$.
\end{defn}
So for all graphs $c_{0,0}(\Gamma) = 1$, $\sum_{j}c_{1,j}(\Gamma) =
\abs{V}$, $\sum_jc_{2,j}(\Gamma) = \abs{E}$ and for $i>j$,
$c_{i,j}(\Gamma)=0$. Also for $K_n$, the complete graph on $n$ vertices,
$\sum_j c_{i,j}(K_n) = \binom{n}{i}$. We make the following definition.

\begin{defn} \label{def:cliquePolynomial}
  The \emph{clique polynomial} of $\Gamma$ is given by 
  \begin{equation*}
    c_\Gamma(z) = \sum_{j=0}^{\infty}\sum_{i=0}^{\infty} (-1)^i c_{i,j}(\Gamma) z^j.
  \end{equation*}
\end{defn}
If for all $i \in V$, $p_i=1$, then $c_{i,j}(\Gamma) = 0$ unless
$i=j$.  So $c_\Gamma(z) = \sum_{i=0}^{\infty} (-1)^i c_{i,i}(\Gamma)
z^i$ in this case.

\subsection{Algebras from graphs} \label{section:algFromGraphs}

\begin{defn} \label{def:A(G)} Given a finite simple weighted graph,
  $\Gamma$ as above, let $\bfk \langle x_1, \ldots, x_{\abs{V}}
  \rangle$ be the associative graded algebra with the degree of $x_i$,
  denoted $\abs{x_i}$, equal to the weight $p_i$. For $j \in E$, let
  $\{a_j,b_j\}$ denote the boundary of $j$.  Let $B(\Gamma)$ denote
  the subalgebra of $\bfk \langle x_1, \ldots, x_{\abs{V}} \rangle$
  generated by $\{[x_{a_j},x_{b_j}], j \in E\}$. Let $I(\Gamma)$ be
  the two-sided ideal generated by $B(\Gamma)$. That is,
\begin{equation*}
I(\Gamma)= (\{[x_{a_j},x_{b_j}] \mid j \in E\}).
\end{equation*}
Define the graded algebra associated to the graph $\Gamma$ to be
  \[
  A(\Gamma) = \bfk \langle x_1, \ldots, x_{\abs{V}} \rangle / I(\Gamma).
  \]
\end{defn}

\begin{rem}
  This can be thought of as a graded version of the graph algebras of
  \cite{km-lnr:graphAlgebras,kimRoush}. We avoid the convenient
  terminology ``graph algebras'' since it has other common usages.
\end{rem}

Define the differential graded algebra associated to $\Gamma$ to be
\begin{equation*}
  \DGA(\Gamma) = 
  \bigl( \bfk \langle x_1, \ldots, x_{\abs{V}}, y_1, \ldots, y_{\abs{E}} \rangle,
  dx_i=0, dy_j = [x_{a_j},x_{b_j}] \bigr),
\end{equation*}
where $\abs{x_i}=p_i$ and $\abs{y_j}=p_{a_j}+p_{b_j}+1$. The differential
reduces the degree by one.

\begin{ex}\label{ex:pentagon}
  Let $\Gamma$ be a pentagon, that is the $5$--cycle graph together
  with weights on its vertices. Then
  \begin{gather*}
    A(\Gamma) = \bfk \langle x_1, x_2,x_3,x_4, x_5 \rangle /
    ([x_1,x_2]
    , [x_2,x_3], [x_3,x_4], [x_4,x_5], [x_5,x_1]), \text{ and}\\
    \DGA(\Gamma) = \bigl( \bfk \langle x_1, \ldots, x_5, y_1, \ldots,
    y_5 \rangle, dx_i=0, dy_j = [x_j,x_{j+1}] \bigr),
  \end{gather*}
where $x_6=x_1$.
\end{ex}

\subsection{Inert ideals}

\begin{lem} \label{lemma:bijection}
  The surjection $\phi: \bfk \langle [x_{a_j},x_{b_j}], \; j \in E \rangle
  \to B(\Gamma)$ from the free associative algebra is a bijection.
\end{lem}

\begin{proof}
  Let $L(\Gamma)$ denote the Lie subalgebra of the free Lie algebra $\freeL
  \langle x_1, \ldots, x_{\abs{V}} \rangle$ generated by
  $\{[x_{a_j},x_{b_j}], \; j \in E\}$. Then there is a surjection
  $\theta: \freeL \langle [x_{a_j},x_{b_j}], \; j \in E \rangle \to
  L(\Gamma)$. Over a field of characteristic not equal to $2$, any
  subalgebra of a free graded Lie algebra is a free graded Lie
  algebra~\cite[Theorem 14.5]{mz:book}. So $\theta$ is an
  isomorphism. Let $U$ denote the universal enveloping algebra functor
  from graded Lie algebras, to graded associative algebras. Then $\phi
  \isom U(\theta)$. Thus $\phi$ is also a bijection.
\end{proof}

\begin{thm}[\cite{anick:inert}] \label{thm:anick}
  The following are equivalent:
\begin{enumerate}
\item \label{anick1}
  $H(\DGA(\Gamma)) \isom A(\Gamma)$,
\item \label{anick2} $ \bfk \langle x_1, \ldots, x_n \rangle \isom
  B(\Gamma) \amalg A(\Gamma) \text{ (as vector spaces),} $
\item \label{anick3} $H_{A(\Gamma)}(z) = \dfrac{1}{1-(z^{p_1} + \ldots z^{p_n}) + (z^{q_1} +
    \ldots + z^{q_m})}, $ and
\item \label{anick4} $A(\Gamma)$ has global dimension $\leq 2$.
\end{enumerate}
\end{thm}

If these equivalent conditions are satisfied we say that the set
$\{[x_{a_j},x_{b_j}] \; , j \in E\}$ is \emph{inert}.  Equivalently,
one says that $I(\Gamma)$ is \emph{inert}.

\begin{proof}
The result follows from combining Lemma~\ref{lemma:bijection} and several results of Anick's. 
By Lemma~\ref{lemma:bijection} and \cite[Theorem 2.6]{anick:inert}, \eqref{anick2} and \eqref{anick3} are equivalent.
By Lemma~\ref{lemma:bijection} and \cite[Theorem 2.9]{anick:inert}, \eqref{anick1} and \eqref{anick2} are equivalent.
By Lemma~\ref{lemma:bijection} and \cite[Theorem 2.12(b)]{anick:inert}, \eqref{anick2} and \eqref{anick4} are equivalent.
\end{proof}

\begin{proof}[Proof of Corollary~\ref{cor:algebra}]
By Theorem~\ref{thm:main}, it is enough to show that
\begin{equation} \label{eqn:notriangles}
  1-(z^{p_1} + \cdots + z^{p_{\abs{V}}}) + (z^{q_1} + \cdots + z^{q_{\abs{E}}}) 
  = c_\Gamma(z)
\end{equation}
if and only if $\Gamma$ does not contain a triangle. The
remainder of Corollary~\ref{cor:algebra} then follows from
Theorem~\ref{thm:anick}.

If there is no triangle in $\Gamma$, then there are no cliques with
more than 2 vertices. So we have the desired equality
\begin{equation*}
  c_{\Gamma}(z) = \sum\limits_{i=0}^{\infty}\sum\limits_{j=0}^{\infty}
  (-1)^i c_{i,j} z^j
  = 1- \sum\limits_{v\in V} z^{\abs{v}}+ \sum\limits_{e\in E} z^{\abs{e}}.
\end{equation*}

On the other hand suppose Equation~\ref{eqn:notriangles} is true.
Then
\begin{equation*}
  \sum\limits_{i=3}^{\infty}\sum\limits_{j=0}^{\infty} (-1)^i c_{i,j} z^j =0
\end{equation*}
So 
\begin{equation}\label{eqn:altsum}
\sum\limits_{i=3}^{\infty} (-1)^i c_{i,j}=0 \;\; \mbox{ for each } j.
\end{equation} 
Assume $\Gamma$ contains a triangle.  Let $j_0$ be the minimal weight
of a triangle in $\Gamma$. Since $j_0$ is minimal and the weights are
positive integers, there cannot be any cliques with more than 3
vertices having weight $j_0$, so $c_{i,j_0}=0$ for $i>3$. But then by
Equation~\ref{eqn:altsum} $c_{3,j_0}=0$.  This contradicts the
existence of a triangle of weight $j_0$.  Therefore, $\Gamma$ does not
have any triangles.
\end{proof}

\section{Noncommutative Gr\"obner bases and Hilbert
  series} \label{section:grobner}

Gr\"obner bases provide a nice generating set for an ideal in both
commutative and noncommutative polynomial rings because they allow us
to compute a normal form for elements in the quotient by the ideal. As
the reader may not be familiar with noncommutative Gr\"obner bases, we
provide a fairly detailed description here. An excellent resource on
noncommutative Gr\"obner bases is the paper by Ufnarovski
\cite{Ufn98}.

Let $R= \bfk \langle x_1, \ldots, x_m\rangle$ be a noncommutative
polynomial ring. To define the Gr\"obner basis of a 2-sided ideal of
$R$ we must first define an ordering on the monomials, or words, of
$R$.

\begin{defn}
  An {\em admissible ordering} on $R=\bfk \langle x_1, \ldots, x_m
  \rangle$ is a relation $\geq$ on the words of $R$ such that
  \begin{enumerate}
    \item $\geq$ is a total ordering on the set of words of $R$,
    \item for all words $f,g,h,k$ in $R$, if $f \geq g$ and $h \geq k$,
      then $fh \geq gk$, 
      and  
    \item every infinite sequence of words $w_1 \geq w_2 \geq w_3 \geq
      \cdots$ eventually stabilizes, i.e. $w_i=w_j$ for all $i,j > i_0$.
  \end{enumerate}
\end{defn}

The familiar lexicographical ordering on a commutative polynomial ring is not
an admissible ordering for noncommutative polynomials since if $a\geq b$
then $a\geq ba\geq b^2a\geq \cdots$. The following degree
lexicographic ordering, however, is admissible.

Suppose the degree of the variable $x_i$ is $d_i$. We say the {\it
  degree} of a word $w=x_{i_1}x_{i_2}\cdots x_{i_r}$ is
$\deg(w)=d_{i_1}+d_{i_2}+ \cdots +d_{i_r}$.

\begin{defn}
  The {\em degree lexicographic ordering}, called {\em DegLex}, is the
  ordering such that for words $w,v$ in $R$, $w>v$ if
  \begin{enumerate}
  \item $\deg(w) > \deg(v)$, or
  \item $\deg(w)=\deg(v)$ and lexicographically $w$ comes before $v$.
  \end{enumerate}
\end{defn}

Notice that one has to specify an order on the variables before using
DegLex order.

Let $f=\sum_i c_i w_i$ be an element of $R$ where $c_i \in
{\bfk}-\{0\}$ and $w_i$ are words in the variables. Given an
admissible ordering $\geq$ on $R$, we call the largest (with respect
to $\geq$) of the $w_i$ the {\em initial term} of $f$ and we denote it
by $in(f)$. For a set $S \subset R$ we will write $\initial(S)$ for
the set of all initial terms of elements of $S$.

\begin{defn}
  Let $I$ be a two-sided ideal of $A$. A subset $G$ of words of $I$ is
  a {\it Gr\"obner basis} of $I$ if the 2-sided ideal generated by
  $\initial(G)$ is equal to the ideal generated by $\initial(I)$.
\end{defn}

So for every $f \in I$, some subword of its initial term
$\initial(f)$ is the initial term of an element of $G$.

A major difference between the commutative and noncommutative cases is
that noncommutative Gr\"obner bases are often infinite. 

To compute Gr\"obner bases, we use the idea of a {\em rewriting rule}
for the noncommutative polynomial $f$ with initial term
$w=\initial(f)$. It is the rule that sends $w$ to $g$ where $f$ is
proportional to $w-g$. We write $w \to_f g$. Note that if the leading
coefficient of $f$ is 1 then $f=w-g$.

For a polynomial $p$ with terms containing $w$ as subwords, applying
the rewriting rule $w \to_f g$ to $p$ means replacing all (successive)
occurrences of $w$ by $g$. If the result of this (successive)
replacement is $q$, then we write $p \to_f q$. Notice that $q$ is
smaller than $p$ in the ordering.

Let $f,g \in R$ and let $u,v$ be their initial terms respectively. A
triple of words $(a,b,c)$ is called a {\em composition} of $f$ and $g$
if $ab=u$ and $bc=v$. If the rewriting rules for $f$ and $g$ are $u
\to_f h$ and $v \to_g k$ then the {\em result of the composition}
$(a,b,c)$ is the difference $ak-hc$. There may be multiple
compositions for the same pair of polynomials. The trivial one (where
$b=1$) always reduces to zero after rewriting.

Let $S$ be a set of polynomials in $R$, and let $f,h \in R$. We say
that $f$ {\em reduces by $S$ to $h$} if $h$ can be obtained from $f$
by applying a sequence of rewriting rules for elements of $S$. We
write $f \Rightarrow_S h$.

\begin{excont} Let $\Gamma$ be the pentagon where each vertex has
  weight $1$.  For $i = 1, \ldots, 4$, let $g_i = [x_i, x_{i+1}]$ and
  let $g_5=[x_5,x_1]$.  Order $\bfk\left<x_1,\ldots x_5\right>$ using
  DegLex order with $x_1>x_3>x_5>x_2>x_4$ and $\deg(x_i)=1$.

  We have $\initial(g_4) = x_5 x_4$ and $\initial(g_5) = x_1 x_5$ and
  rewriting rules $x_5 x_4 \to_{g_4} -x_4 x_5$ and $x_1 x_5 \to_{g_5}
  -x_5 x_1$.  There are two compositions of $g_5$ and $g_4$:
  $(x_1x_5,1,x_5x_4)$ and $(x_1,x_5,x_4)$. The result of the first
  composition is 
  \begin{equation*}
  \begin{split}
    (x_1x_5)(-x_4x_5)-(-x_5x_1)(x_5x_4) &= - x_1 x_5 x_4 x_5 + x_5 x_1
    x_5 x_4\\
    &\to_{g_5} x_5 x_1 x_4 x_5 + x_5 x_1 x_5 x_4\\
    &\to_{g_4} x_5 x_1 x_4 x_5 - x_5 x_1 x_4 x_5 = 0.
  \end{split}
  \end{equation*}
  So the result reduces to zero by the set $S = \{g_1, \ldots g_5 \}$. 

  The result of the second composition is $(x_1)(-x_4x_5)-(-x_5x_1)(x_4)
  = -x_1 x_4 x_5 + x_5 x_1 x_4$, which does not reduce to zero by $S$.
\end{excont}

The reduction process appears to depend on the order of the rewriting
rules, but if the set $S$ is a Gr\"obner basis, then there is a unique
minimal reduction of $f$ by $S$. In fact, the converse is true as
well.
 
\begin{thm}[Bergman's Diamond Lemma, \cite{Ber78}]
  Let $G$ be a self-reduced set, that is, no element of $G$ can be
  further reduced by $G$. $G$ is a Gr\"obner basis if and only if the
  results of all possible compositions of elements of $G$ reduce by
  $G$ to zero.
\end{thm}

\begin{excont}
  The set $S= \{g_1, \ldots, g_5\}$ in the pentagon case above is not
  a Gr\"obner basis because the result of the composition
  $(x_1,x_5,x_4)$ of $g_4$ and $g_5$ does not reduce by $S$ to
  zero. One can check, however, that
  \[
  G = \{g_1, \ldots, g_5\} \cup \{h_k\}_{k=1}^{\infty} \mbox{ where }
  h_k = x_1x_4^k x_5 + (-1)^k x_5x_1x_4^k
  \]
  is a Gr\"obner basis.
\end{excont}

Bergman's Diamond Lemma implies a noncommutative analogue of
Buchberger's algorithm for finding a Gr\"obner basis, known as Mora's
algorithm \cite{Mor86}:

Let $G_0$ be a self-reduced generating set for the ideal. Create an
ordered list of all the possible compositions of elements of
$G_0$. Work through the list of compositions one at a time, and if one
is found whose result does not reduce by $G_0$ to zero, then append it
to the set of generators and self-reduce to get a new generating set
$G_1$. Adjust the list of compositions accordingly. If there is a
finite Gr\"obner basis, then eventually the list of compositions will
be empty and the final $G_k$ will be the Gr\"obner basis. Otherwise
there will be an infinite number of larger and larger degree
compositions to consider. 

Computer programs such as BERGMAN \cite{Bergman} compute the Gr\"obner
basis up to a fixed degree. In the above example, we see a pattern in
the new generators and hence guess the form of the Gr\"obner basis.

 \bigskip

We recall a well-known, though not obvious, fact about Hilbert series
which is crucial to the proof of Theorem~\ref{thm:main}
\begin{prop} \label{lemma:grobnerY} 
  Let $I$ be a 2-sided homogeneous ideal in $R = \bfk\langle x_1,
  \ldots ,x_m\rangle$. Then $ H_{R/I}(z) = H_{R/ \initial(I)}(z). $
\end{prop}

\begin{proof}
  We want to show that $\dim_\bfk (R/I)_n = \dim_\bfk (R/\initial(I))_n$
  for all $n \geq 0$. It is equivalent to show that the dimensions of
  $I_n$ and $(\initial(I))_n$ are the same. Suppose $I_n$ has a vector
  space basis $f_1, \ldots , f_p$ of monic polynomials. Let
  $w_i=\initial(f_i)$. If $w_i=w_j$ for some $i\neq j$, then replace
  $f_i$ with $f_i -f_j$.  Thus, we may assume that $w_1 > w_2 > \cdots
  > w_p$. So the $w_1, \ldots , w_p$ are linearly independent.  We
  claim that $\{w_1, \ldots , w_p\}$ is a basis for $(\initial(I))_n$.

  Suppose $h \in (\initial(I))_n$. Then $h = \sum_i r_i u_i
  \initial(t_i) v_i$, where $r_i \in \bfk$, $t_i \in I$ and
  $u_i, v_i$ words in $x_1, \ldots, x_n$. Since $\geq$ is admissible, $u_i
  \initial(t_i) v_i = \initial (u_i t_i v_i)$. Let $s_i = u_i t_i
  v_i$. Since $s_i \in I_n$, it follows that $s_i$ is a linear
  combination of $\{f_1, \ldots, f_p\}$. Therefore $\initial(s_i) =
  w_j$ for some $j \in \{1, \ldots, p\}$. Thus $h$ is a linear
  combination of $\{w_1,\ldots,w_p\}$, and so $\{w_1, \ldots, w_p\}$
  is a basis for $(\initial(I))_n$.
\end{proof}

\section{Partially commutative monoids} \label{section:trace}

A \emph{free partially commutative monoid} is a monoid of the form
\begin{equation*}
  M = \langle x_1, \ldots x_n , \rangle / \simeq_I,
\end{equation*}
where $I$ is a set of pairs in $x_1, \ldots, x_n$ and $\simeq_I$ is
the congruence relation generated by setting $ab=ba$ for all $\{a,b\}
\in I$.  We will let $w$ denote an element of $M$ and $1$ denote the
unit of $M$.

Then $M$ can be represented by a finite simple graph $\Gamma$ whose
vertices correspond to the generators $x_1, \ldots, x_n$ of the
monoid, and whose edges correspond to the elements of $I$.  The monoid
$M$ is also called a trace monoid or a monoid of circuits on a graph.
Mazurkiewicz~\cite{mazurkiewicz} introduced these monoids to the study
concurrent systems, where the generators of $M$ correspond to
processes and $I$ lists processes that are independent. Surveys of the
subject include \cite{diekertMetivier}, \cite{diekertRozenberg} and
\cite{lallement}.

Cartier and Foata~\cite{cartierFoata} studied the combinatorics of
partially commutative monoids and proved a version of
Theorem~\ref{thm:main}, which we will now describe.

Recall from Definition~\ref{def:fps}, that a \emph{formal power
  series} on $M$ is a function from $M$ to $\ZZ$.  For example, for
$S \subset M$, the characteristic function $\chi_S$, given by
\begin{equation} \label{eq:characteristicFunction}
  \chi_S(w) = 
  \begin{cases}1 & \text{if } w \in S\\ 0 & \text{if } w \notin S 
  \end{cases}
\end{equation}
is a formal power series. An important special case is the
characteristic function $\chi_M$ which is the constant function
$1$. The set of all power series on $M$ is a ring denoted by $\ZM$.
The unit of this ring, $\bf1$, is given by $\chi_{\{1\}}$.  An element $f
\in \ZM$ is invertible if and only if $f(1)=\pm 1$.  For $f \in \ZM$,
$f = \sum_{w \in M} f(w) \chi_{\{w\}}$. As is customary, we will write
$f$ as $\sum_{w \in M} f(w) w$.

Say that $Q = \{ x_{i_1}, x_{i_2}, \ldots, x_{i_l}\} \subset \{x_1,
\ldots, x_n\}$ is a \emph{clique} of $M$ if the corresponding vertices
of $\Gamma$ form a clique. Since the elements in $Q$ commute, writing
these elements in any order we obtain a representative for the same
unique word $w_Q$ in $M$. Therefore for each clique $Q$, there is a
unique $[Q] \in \ZM$ given by
\begin{equation*}
  [Q] = \chi_{\{w_Q\}},
\end{equation*}

Let $\mathcal{Q}$ denote the finite set of all cliques of $M$
(including the empty clique).  Define the \emph{clique polynomial} of
$M$ to be the formal power series
\begin{equation} \label{eq:cliqueM}
  \mu_M = \sum_{Q \in \mathcal{Q}} (-1)^{\abs{Q}} [Q].
\end{equation}

\begin{ex}
  Let $M$ be the monoid generated by $a,b,c,d,e$ with relations
  $ab=ba$, $ac=ca$, $bc=cb$, $cd=dc$. Then the following graph
  represents $M$.
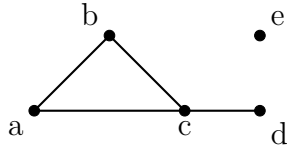
\begin{center}
\begin{tikzpicture}
\draw[thick] (-1,0)--(1,0);
\draw[thick] (-1,0)--(0,1);
\draw[thick] (0,1)--(1,0);
\draw[thick] (1,0)--(2,0);
\filldraw[black] (-1,0) circle (2pt) node[anchor=north east]{a}
(0,1) circle (2pt) node[anchor=south east]{b}
(1,0) circle (2pt) node[anchor=north ]{c}
(2,0) circle (2pt) node[anchor=north west]{d}
(2,1) circle (2pt) node[anchor=south west]{e};
\end{tikzpicture}
\end{center}

The clique polynomial of $M$ is
\[
\begin{split}
  \mu_M =& \chi_{\phi}-\chi_{\{a\}} - \chi_{\{b\}} -\chi_{\{c\}} -\chi_{\{d\}} -\chi_{\{e\}} \\
  & +\chi_{\{ab\}}+\chi_{\{ac\}}+\chi_{\{bc\}} +\chi_{\{cd\}}
 -\chi_{\{abc\}}\\
=&1-a-b-c-d-e+ab+ac+bc+cd-abc.
\end{split}
\]
\end{ex}

\begin{thm}[Cartier and
  Foata~\cite{cartierFoata}] \label{thm:cartierFoata} The clique
  polynomial of $M$ is the inverse of the constant power series
  $\chi_{M}$.
\end{thm}

We recall Cartier and Foata's elegant proof.

\begin{proof}
  Let $w \in M$. We want to show that $\mu_M \cdot \chi_M =
  \chi_{\{1\}}=\bf1$. That is, $\mu_M \cdot \chi_M (w) = \begin{cases}
    1 & \text{ if }w = 1\\ 0 & \text{otherwise} \end{cases}$.  By the
  definition of the Cauchy product in $\ZM$,
  \begin{equation*}
    \mu_M \cdot \chi_M (w) = 
    \sum_{\substack{uv=w\\u = [Q], Q \in \mathcal{Q}}} (-1)^{\abs{u}}.
  \end{equation*}
  If $w = 1$ then $w$ can be uniquely expanded as $1\cdot 1$, where
  the first $1$ is the word corresponding to the $0$--clique. So
  $\mu_M \cdot \chi_M (1) = 1$.  Assume $w \neq 1$.  Let $S = \{ a \in
  \{x_1, \ldots , x_n\}\ | \ w = a w', \ w' \in M\}$.  Then $S \in
  \mathcal{Q}$. Let $m = \abs{S}$. Therefore,
  \begin{equation*}
    \mu_M \cdot \chi_M (w) = \sum_{T \subset S} (-1)^{\abs{T}} = 
    \sum_{k=0}^m \binom{m}{k} (-1)^k = (1-1)^k = 0. \qedhere
  \end{equation*}
\end{proof}

\section{Proof of the main theorem}
\label{section:proof}

In this section we prove a generalization of Theorem~\ref{thm:main},
in which we assign weights to both the vertices and the edges of the
graph.

Let $\Gamma$ be a finite simple graph with vertices $V$ and edges $E$
in which each vertex $i \in V$ is labeled with a weight $p_i \in \N$
and each edge $j \in E$ is labeled with a weight $q_j \in
\Zmodtwo$. Let $c_{\Gamma}(z)$ be the clique polynomial of $\Gamma$ as
in Definition~\ref{def:cliquePolynomial}.

Let $R = \bfk \langle x_i \;, i \in V \rangle$, with $\abs{x_i} =
p_i$.  For $j \in E$, let $[x_{a_j},x_{b_j}]$ denote $x_{a_j} x_{b_j} -
(-1)^{q_j}x_{b_j} x_{a_j}$.  Let $I(\Gamma)$ be the two--sided ideal in $R$
generated by $\{[x_{a_j},x_{b_j}] \;, j \in E\}$.

\begin{thm} \label{thm:generalized}
The associative (noncommutative) graded algebra
 \[
 A = R / I(\Gamma)
  \]
  has Hilbert series
  \[
  H_A(z) = \left[c_{\Gamma}(z)\right]^{-1}.
  \]
\end{thm}

Theorem~\ref{thm:main} follows from this theorem by taking $q_j =
p_{a_j} p_{b_j} \mod 2$.  We start with the case in which all of the edges
have weight $0$.

\begin{prop}
  The result of Theorem~\ref{thm:generalized} holds if $q_j=0$ for
  all $j \in E$.
\end{prop}

\begin{proof}
  Let $M$ be the partially commutative monoid $\langle x_i \; , i \in
  V \rangle / \simeq_E$.  For $w \in M$, $w\neq 1$, let $\weight(w) =
  p_{i_1} + \cdots + p_{i_l}$ where $x_{i_1} \cdots x_{i_l}$ is some
  representative of $w$ and $\weight(1)=0$. This is well--defined,
  since for all $j \in E$, $\weight(x_{a_j} x_{b_j}) = \weight (x_{b_j}
  x_{a_j})$. Then $M$ is a basis for $A$ as a $\bfk$--vector space, and
  furthermore the product in $A$ restricts to the composition product
  on $M$.

Define $\Weight:\ZM \to \Zz$ by 
\begin{equation*}
  \Weight\left(\sum_{w \in M}f(w)w\right) = \sum_{w \in M}f(w)z^{\weight(w)}.
\end{equation*}
For $S \subset M$ let $\chi_S$ denote the characteristic function,
defined in~\eqref{eq:characteristicFunction}.  Then for $\chi_M$, the
identity on $M$,
\begin{equation*}
\begin{split}
  \Weight(\chi_M) &= \sum_{k=0}^{\infty} \alpha_k z^k, \text{ where }
  \alpha_k = \#\{w \ | \ \weight(w) = k\}\\
  &= H_A(z),
\end{split}
\end{equation*}
and
$\Weight(\chi_{\{1\}}) = 1$.

By Lemma~\ref{lem:ringhom} below, $\Weight$ is a ring homomorphism.

Let $\mu_M$ be the clique polynomial for $M$, defined in~\eqref{eq:cliqueM}.
Then
\begin{equation*}
  \begin{split}
    \Weight(\mu_M) &=  \sum_{Q \in \mathcal{Q}} (-1)^{\abs{Q}} z^{\weight([Q])}\\
      &= \sum_{i=0}^{\infty} \sum_{j=0}^{\infty} (-1)^i c_{i,j} z^j\\
      &= c_{\Gamma}(z),
  \end{split}
\end{equation*}
where $c_{i,j}$ equals the number of $i$--cliques in $M$ of weight $j$.

Then, by Theorem~\ref{thm:cartierFoata}, $H_A(z) \cdot c_{\Gamma}(z) =
\Weight(\chi_M) \cdot \Weight(\mu_M) = \Weight(\chi_M \cdot \mu_M) =
\Weight(\chi_{\{1\}}) = 1$.
\end{proof}

\begin{lem}\label{lem:ringhom}
 $\Weight: \ZM \to \Zz$ is a ring homomorphism.
\end{lem}

\begin{proof}
  \begin{multline*}
    \Weight(f + g) = \sum_{w \in M}(f+g)(w)z^{\weight(w)} = \\
    = \sum_{w\in M}f(w)z^{\weight(w)} + \sum_{w\in
      M}g(w)z^{\weight(w)} = \Weight(f) + \Weight(g).
  \end{multline*}
  Let $M_k$ denote the subset of $M$ of words with weight $k$.
  \begin{gather*}
    \Weight(f\cdot g) = \sum_{k=0}^{\infty} c_k z^k, \text{ where }
    c_k = \sum_{w \in M_k} \sum_{uv=w} f(u)g(v)\\
    \Weight(f)\Weight(g) = \sum_{k=0}^{\infty} d_k z^k, \text{ where }
    d_k = \sum_{i+j=k} \sum_{u \in M_i} f(u) \sum_{v \in M_j} g(v).
  \end{gather*}
  That $c_k = d_k$ follows from remarking that 
  \begin{equation*}
    \#\{(u,v) \in M \times M \ | \ uv \in M_k\} = 
    \# \{ (u,v) \in M_i \times M_j \ | \ i+j=k\}. \qedhere
  \end{equation*}
\end{proof}

To prove Theorem~\ref{thm:generalized} it remains to show that the
choice of $q_j \in \Zmodtwo$ does not affect $H_A(z)$. We will use
noncommutative Gr\"obner bases to prove this.

Let $\geq$ be an admissible ordering on $R=\bfk\left<x_1, \ldots,
x_{\abs{V}}\right>$.  Let $G$ be a Gr\"obner basis for $I(\Gamma)$.  Since
$H_{R/I}(z) = H_{R/\initial(I)}(z)$ and $R/ \initial(I) = R / (\initial(G))$,
it remains to show that one can choose a Gr\"obner basis for $I(\Gamma)$
such that $\initial(G)$ does not depend on the $q_j$.

For simplicity, enumerate the list $\{[x_{a_j}, x_{b_j}], \; {j \in E}\}$
by $\{g_1, \ldots, g_m\}$.  Say that a rewriting rule $u \to_f v$ is a
\emph{zig--zag of elementary rewrites} if it can be written as
a sequence of rewrites
\begin{equation*}
  u \to_{g_{i_1}} u_1\text{ }_{g_{i_2}}\!\!\!\from u_2 \to_{g_{i_3}} 
  \ldots \text{ }_{g_{i_{n-1}}}\!\!\!\!\!\!\!\from u_{n-1} \to_{g_{i_n}} v.
\end{equation*}

\begin{lem}
  Any zig--zag of elementary rewrites from a word $w$
  to $\pm w$, has an even number of $g_i$ for each $i$.
\end{lem}

\begin{proof}
  Fix $i$ and $\Phi$, a zig--zag of elementary rewrites from $w$ to
  $\pm w$. Assume $x_{a_i}x_{b_i} \to_{g_i} (-1)^{q_i}x_{b_i}x_{a_i}$.
  Assume $w = \alpha_1 \cdots \alpha_s$, $\alpha_l \in \{x_1, \ldots,
  x_n\}$.  Let $\beta_1 \cdots \beta_t$ be the word obtained from $w$ by
  deleting all letters except $x_{a_i}$ and $x_{b_i}$.  Set $P(w) = \#
  \{k \ | \ \beta_{2k} = x_{a_i}\} \mod 2$.  Then rewriting using $g_j$
  changes $P$ if and only if $i=j$.  Therefore $\Phi$ contains an even
  number of $g_i$.
\end{proof}

\begin{lem} \label{lemma:noZigZag}
  There does not exist a zig--zag of elementary rewrites from $w$ to $-w$.
\end{lem}

\begin{proof}
  Let $\Phi$ be a zig--zag of elementary rewrites from $w$ to $\pm
  w$. Let $\alpha_i$ be the number of $g_i$ in $\Phi$. Then let $\alpha =
  \sum_{i=1}^{\abs{V}} \alpha_i q_i$. Then $\Phi$ goes from $w$ to $(-1)^\alpha
  w$. By the previous lemma, $\alpha_i$ is even for all $i$, so $(-1)^\alpha =
  1$.
\end{proof}

The following proposition completes the proof of
Theorem~\ref{thm:generalized}.

\begin{prop}
  The two--sided ideal $I(\Gamma)$ has a Gr\"obner basis $G$ such that
  \begin{enumerate}
  \item all of the elements of $G$ have rewriting rules which are zig--zags
    of elementary rewrites, and
  \item $\initial(G)$ does not depend on the weights $\{q_k \; , k \in E\}$.
  \end{enumerate}
\end{prop}

\begin{proof}
  The proof is by induction on the sets $G_i$ arising during the
  application of Mora's algorithm to find the Gr\"obner basis of
  $I(\Gamma)$. Recall that at each step where we obtain a nonzero
  result $r_i$, $G_{i+1}$ is the self reduction of $G_i \cup {r_i}$.

  We start with $G_0 = \{g_1, \ldots, g_m\}$, all of whose elements
  have elementary rewriting rules and whose initial terms do not
  depend on the weights $\{q_k \; , k \in E\}$.

  Suppose (1) and (2) of the Proposition are true for $G_i$ and that
  $f,g \in G_i$. So we have two rewriting rules $ab \to_f h$ and $bc
  \to_g k$ both of which are zig--zags of elementary rewriting rules.
  Then $f$ and $g$ have composition $(a,b,c)$ which has result $r =
  ak-hc$. If $r\neq 0$, then we have either $ak \to_r hc$ or $hc \to_r
  ak$. In either case this is a zig--zag of $f$ and $g$,
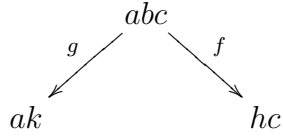
  \begin{equation*}
    \xymatrix{& abc \ar[dl]_g \ar[dr]^f\\ ak & & hc}
  \end{equation*}
  which, by induction, is a zig--zag of elementary
  rewrites. 

  If $r \neq 0$ then we add $r$ to the Gr\"obner basis. Self-reducing
  $G_i \cup {r}$ may change the elements, but they will still have
  rewriting rules that are zig-zags of elementary rewrites so all of
  the elements of $G_{i+1}$ will have rewriting rules which are
  zig-zags of elementary rewrites. We remark that $\initial(r)$ does
  not depend on $\{q_k \; ,k \in E\}$. 

  It remains to show that it cannot be that $r=0$ for one choice of
  $\{q_k \; , k \in E\}$, and $r \neq 0$ for another choice. Fix a
  choice of $\{q_k \; , k \in E\}$, and assume that $r=ak-hc=0$. This implies
  that there are sequences of elementary rewrites from both $ak$ and
  $hc$ to the same word $u$. 
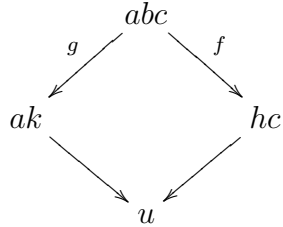
  \begin{equation*}
    \xymatrix{& abc \ar[dl]_g \ar[dr]^f\\ ak \ar[dr] & & hc \ar[dl]\\ & u}
  \end{equation*}
  Composing, we have a zig--zag, $\Phi$, of elementary rewrites from
  $u$ to $u$.  Notice that for any zig--zag of elementary rewrites,
  changing the choice of $\{q_k \; , k \in E\}$ only changes the signs of
  the terms in the zig--zag.  By Lemma~\ref{lemma:noZigZag}, the
  resulting zig--zag is still from $u$ to $u$. That is, we still have
  $r=0$.
\end{proof}

\begin{excont}
  When $\Gamma$ is the pentagon, the clique polynomial is
  $1-5z+5z^2$. So $H_{A(\Gamma)}(z) = 1/(1-5z+5z^2)$.  Furthermore by
  Corollary~\ref{cor:algebra}, $I(\Gamma)$ is inert, as is the
  attaching map in $(S^{p_i+1},*)^{\Gamma}$, and $A$ has global
  dimension $2$.
\end{excont}

\section{Graph products of spheres and Adams--Hilton
  models} \label{section:topology}

Given a simply--connected CW complex, $Y,$ a useful algebraic model is
the Adams--Hilton model~\cite{adamsHilton}. It is a free differential
graded algebra (DGA), $\AH(Y)$, whose algebra generators are in 1--1
correspondence with the cells of $Y$. Furthermore, there is a morphism
of DGAs from $\AH(Y)$ to the singular chain complex on the Moore loops
on $Y,$ that induces an isomorphism $H\AH(Y) \isom H_*(\Omega Y;
\bfk)$, where $\Omega Y$ denotes the space of pointed loops on $Y$.
For a nice summary of the properties of Adams--Hilton models
see~\cite[Theorem 11.10.7]{selick:book}.

Let $\underline{X}^{\Gamma}$ be the generalized moment--angle complex
defined in Section~\ref{section:introduction} before the statement of
Theorem~\ref{thm:topology}. Then $\underline{X}^{\Gamma}$ has an
Adams--Hilton model~\cite{adamsHilton} given by the differential
graded algebra 
\begin{equation*}
  \AH(\underline{X}^{\Gamma}) 
  = (\bfk \langle x_1, \ldots, x_n, y_1, \ldots, y_m \rangle, d), 
\end{equation*}
\begin{equation*}
  \text{ where } \abs{x_i} = p_i, \; \abs{y_j} = q_j+1, \text{ and }
  dy_j = [x_{a_j}, x_{b_j}].
\end{equation*}
We remark that $\AH(\underline{X}^{\Gamma})$ is a sub--DGA of
$\AH(\underline{X}^K)$ for any simplicial complex $K$ containing
$\Gamma$. 

If we give $\AH(\underline{X}^{\Gamma})$ a grading by setting each
$x_i$ to have degree $0$ and each $y_j$ to have degree $1$, then the
degree $0$ component of the homology $H\AH(\underline{X}^{\Gamma})$ is
\begin{equation*}
  A = \bfk \langle x_1, \ldots, x_n \rangle / I \text{, where } 
  I = ([x_{a_j}, x_{b_j}] \; , j=1 \ldots m).
\end{equation*}

\begin{proof}[Proof of Theorem~\ref{thm:topology}]
  Label the vertices of $\Gamma$ with $\{p_i\}$.  Then using notation
  from Section~\ref{section:algFromGraphs},
  $\AH(\underline{X}^{\Gamma}) = \DGA(\Gamma)$, $I=I(\Gamma)$ and $A =
  A(\Gamma)$.  Furthermore, using the notation of
  Theorem~\ref{thm:topology}, $H_*(\Omega Y;\bfk) \isom A$,
  $H_*(\Omega W;\bfk) \isom \bfk \langle x_1, \ldots, x_n \rangle$,
  and $I(f) = I$.  We have a cofibration $Z \xrightarrow{f} W
  \xrightarrow{i} Y$.

  Statement (3) from Theorem~\ref{thm:topology} can be rewritten as
  $H_*(\Omega Y;\bfk) \isom H_*(\Omega W; \bfk) / I(f)$. By
  Theorem~\ref{thm:anick}, statements (3), (4), (5), and (6) of
  Theorem~\ref{thm:topology} are equivalent.  Statements (1) and (6)
  of Theorem~\ref{thm:topology} are equivalent by
  Theorem~\ref{thm:main}. Equivalence with (2) is provided by F\'elix
  and Thomas~\cite[Theorem 1]{felixThomas:attach}.
\end{proof}

\bigskip

\noindent
\textbf{Acknowledgments.}  The authors would like to thank Yves
F\'elix for introducing them to the generalized moment--angle
complex. The genesis of this paper was a series of excellent
discussions with Yves F\'elix, Greg Lupton and John Oprea. We would
also like to thank Yves F\'elix, Greg Lupton, John Oprea and Jonathan
Scott for their detailed comments on an earlier draft.



\end{document}